\documentclass[a4paper,12pt]{amsart}
\usepackage{amssymb,amsfonts,amsmath}
\usepackage{ifthen}
\usepackage{graphicx}
\usepackage{float}
\usepackage{latexsym, amsthm, amscd, euscript, float, subfig}
\newtheorem{theorem}{Theorem}[section]
\newtheorem{corollary}[theorem]{Corollary}
\newtheorem{lemma}[theorem]{Lemma}

\newtheorem{conjecture}[theorem]{Conjecture}

\theoremstyle{definition}

\newtheorem{remark}[theorem]{Remark}

\numberwithin{equation}{section}

\newcounter{minutes}
\divide\time by 60
\newcounter{hours}
\multiply\time by 60
\addtocounter{minutes}{-\time}

\newcommand{\bpf}{\begin{pf}}
\newcommand{\epf}{\end{pf}}
\newcommand{\beqq}{\begin{eqnarray*}}
\newcommand{\eeqq}{\end{eqnarray*}}

\begin{document}

\bibliographystyle{amsplain}

\title
{Interpolation on Gauss hypergeometric functions with an application}

\def\thefootnote{}
\footnotetext{ \texttt{\tiny File:~\jobname .tex,
          printed: \number\day-\number\month-\number\year,
          \thehours.\ifnum\theminutes<10{0}\fi\theminutes}
} \makeatletter\def\thefootnote{\@arabic\c@footnote}\makeatother

\author[H. M. Arora]{Hina Manoj Arora}
\address{Hina M. Arora, Discipline of Electrical Engineering,
Indian Institute of Technology Indore, Simrol, Khandwa Road,
Indore 453 552, India. \newline
\hspace*{0.45cm}Present Address: MS student,
Department of Applied Mathematics \& Statistics,
Stony Brook, NY-11794-3600}
\email{hina.arora.256@gmail.com}
\email{hina.arora@stonybrook.edu}

\author[S. K. Sahoo]{Swadesh Kumar Sahoo$^*$}
\address{Swadesh Kumar Sahoo, Discipline of Mathematics,
Indian Institute of Technology Indore, Simrol, Khandwa Road,
Indore 453 552, India.}
\email{swadesh@iiti.ac.in}

\thanks{${}^*$ The corresponding author}

\begin{abstract}
In this paper, we use some standard numerical techniques to approximate the hypergeometric function 
$$
{}_2F_1[a,b;c;x]=1+\frac{ab}{c}x+\frac{a(a+1)b(b+1)}{c(c+1)}\frac{x^2}{2!}+\cdots
$$
for a range of parameter triples $(a,b,c)$ on the interval $0<x<1$. 
Some of the familiar hypergeometric functional identities and asymptotic behavior of the hypergeometric function 
at $x=1$ play crucial roles in deriving the formula for such approximations.
We also focus on error analysis of the numerical approximations leading to monotone properties of 
quotient of gamma functions in parameter triples $(a,b,c)$. 
Finally, an application to continued fractions of Gauss is discussed followed by concluding remarks
consisting of recent works on related problems.
\\
\vspace*{-0.2cm}


\end{abstract}

\maketitle

\pagestyle{myheadings}
\markboth{H. M. Arora and S. K. Sahoo}{Interpolation on Gauss hypergeometric functions}

\section{Introduction and Preliminaries}
For a complex number $z$ and $c\neq 0,-1,-2,-3,\ldots$,
the {\em hypergeometric series }is defined by:
$$ 
1+ \sum_{n=1}^\infty
\frac{(a)_n(b)_n}{(c)_n(1)_n} z^n.
$$
Here $(a)_n$ denotes the shifted factorial notation defined, in terms of the gamma function,
by:
$$(a)_n=\frac{\Gamma(a+n)}{\Gamma(a)}=\left\{
\begin{array}{ll}a(a+1) \cdots (a+n-1) & \mbox{if $n\ge 1$};\\
1 & \mbox{if $n=0$, $a\neq 0$.}\end{array}\right.
$$

Note that the hypergeometric series defines an analytic function,
denoted by the symbol $_2F_1[a,b;c;z]$, in $|z|<1$. 
As quoted in the historial remarks in \cite[1.55, p.~24]{AVV97}, the concept of 
hypergeometric series was first introduced by J. Wallis in 1656 to refer to a generalization of the geometric series.
Less than a century later, Euler extensively studied the analytic properties of
the hypergeometric function and found, for instance, its integral representation (see \cite[Theorem 1.19 (2)]{AVV97}. 
Gauss made his first contribution to the subject in 1812. Due to the outstanding contribution made by Gauss to the field,
the hypergeometric function is also sometimes known as
the {\em Gauss hypergeometric function}. Most elementary functions which are solutions to
certain differential equations, can be written in terms of 
the Gauss hypergeometric functions.
One can easily verify by using Frobenius technique that the function $_2F_1[a,b;c;z]$
is one of the solutions of the {\em hypergeometric differential equation} \cite{AAR99,BW10,RV60} 
$$z(1-z)w''+(c-(a+b+1)z)w'-abw=0.
$$
We refer to \cite{RV43,RV60} for Kummer's 24 solutions to the hypergeometric 
differential equation, and to \cite{BW10} for related applications.
The asymptotic behavior of $_2F_1[a,b;c;z]$ near $z=1$ reveals that:
\begin{equation}\label{eq2}
_2 F_1[a,b;c;1]= \frac{\Gamma(c-a-b) \Gamma(c)}{\Gamma(c-a) \Gamma(c-b)}<\infty,
\quad \mbox{valid for ${\rm Re}\,(c-a-b)>0$}.
\end{equation}

Interpolating polynomials for elementary real functions such as trigonometric functions, logarithmic
function, exponential function, etc. have already been derived in undergraduate texts
in Numerical Analysis; see for instance \cite{A89}. These elementary functions are in fact
hypergeometric functions with specific parameters $a,b,c$ (see for instance \cite{AAR99,RV60}). 
Most of such polynomial approximations
are computed when the functional values at the given boundary points are possible.
Hence the asymptotic behaviour (\ref{eq2}) of the hypergeometric function near $z=1$
motivates us to construct interpolating polynomials for real hypergeometric functions 
$_2F_1[a,b;c;x]$, $a,b,c\in \mathbb{R}$, $c\not\in\{0,-1,-2,-3,\ldots\}$,
of a real variable $x$ using several numerical techniques in the interval $[0,1]$, however, 
the interval may be extended to $[-1,1]$ as the hypergeometric series 
in $x$ is convergent for $|x|<1$ and it has a certain asymptotic behaviour near
$-1$ as well with suitable choices of the parameters $a,b,c$;
see for instance \cite[Theorem~26]{RV60}. More precisely, when we compute 
an interpolating polynomial $p_n(x)$ of a hypergeometric function ${}_2F_1[a,b;c;x]$ on $[0,1]$ we take
the value ${}_2F_1[a,b;c;1]$ in the sense that the hypergeometric function defined at
$x=1$ by means of its asymptotic behavior at $x=1$ (see \eqref{eq2}).
Several hypergeometric functional identities also play a crucial role in determining
functional values at the interpolating points.

The following lemmas are useful in describing the error analysis for the interpolating polynomials
that we obtained in this paper. Our subsequent paper(s) in this series will cover the study of 
interpolating polynomials using other techniques. 
 
\begin{lemma}\cite[Lemma~1.33(1), p.~13 (see also Lemma~1.35(2))]{AVV97}\label{lem-AVV}
If $a,b,c\in(0,\infty)$, then $_2F_1[a,b;c;x]$ is strictly increasing on $[0,1)$. In particular,
if $c > a + b$ then for $x \in [0, 1]$ we have 
$$
{}_2F_1[a,b;c;x]\le \frac{\Gamma(c)\Gamma(c-a-b)}{\Gamma(c-a)\Gamma(c-b)}.
$$
\end{lemma}

\begin{lemma}\cite[Lemma~2.16(2), p.~36]{AVV97}\label{lem-AVV-p36}
The gamma function $\Gamma(x)$ is a log-convex function on $(0,\infty)$. In other words,
the logarithmic derivative, $\Gamma'(x)/{\Gamma(x)}$, of the gamma function is increasing 
on $(0,\infty).$
\end{lemma}

Note that in all the plots in this paper, blue color graphs are meant for the original functions 
and red color graphs are for interpolating polynomials.

\section{Linear Interpolation on ${}_2F_1[a,b;c;x]$}\label{sec2}
For performing linear interpolation of the function ${}_2F_1[a,b;c;x]=f(x)$, we consider the 
end points $x_0=0$ and $x_1=1$ of the interval $[0,1]$.
The functional values at these points are respectively $f(0)=1$ and 
$f(1)$ described in \eqref{eq2}.
Hence, the equation of the segment of the straight line joining $0$ and $1$ is
$$P_l(x)= f(x_0)+\frac{x-x_0}{x_1-x_0}(f(x_1)-f(x_0))
=\frac{\Gamma(c) \Gamma(c-a-b) -\Gamma(c-a) \Gamma(c-b)}{\Gamma(c-a) \Gamma(c-b)}x + 1,
$$
when $c-a-b>0$ and $c\neq 0,-1,-2,-3,\ldots$. The polynomial $P_l(x)$ represents
the linear interpolation of ${}_2F_1[a,b;c;x]$ interpolating at $0$ and $1$.

Using Lemma~\ref{lem-AVV},
we obtain the following error estimate:

\begin{lemma}\label{2lem}
Let $a,b,c\in(-2,\infty)$ with $c-a-b>2$.
The deviation of the given function $f(x)={}_2F_1[a,b;c;x]$ from the approximating
function $P_l(x)$ for all values of $x\in[0,1]$ is estimated by
$$|E_l(f,x)|=|f(x)-P_l(x)|\le \frac{|a(a+1)b(b+1)|}{8}
\frac{\Gamma(c)\Gamma(c-a-b-2)}{\Gamma(c-a)\Gamma(c-b)}.
$$
\end{lemma}
\begin{proof}
It requires to maximize 
$$|E_l(f,x)|=\frac{x(1-x)}{2}|f''(x)|
$$
in $[0,1]$, equivalently, to find
$$\frac{(1-0)^2}{8}\max_{0\le x\le 1}|f''(x)|,
$$
where $f(x)={}_2F_1[a,b;c;x]$. The following well-known derivative formula is useful:
\begin{equation}\label{der-for}
\frac{d}{dx}\,{}_2F_1[a,b;c;x]=\frac{ab}{c}\,{}_2F_1[a+1,b+1;c+1;x].
\end{equation}
The proof follows from (\ref{eq2}), Lemma~\ref{lem-AVV}, (\ref{der-for}), and the fact that
$\Gamma(x+1)=x\Gamma(x)$.
\end{proof}

\begin{remark}
It follows from Lemma~\ref{2lem} that there is no error for either of the choices 
$a=0$, $a=-1$, $b=0$, $b=-1$. In other words, for either of these choices
$E_l(f,x)$ vanishes.
\end{remark}

Figure~\ref{Pl} shows linear interpolation of the hypergeometric function at $0$ and $1$, whereas
Table~\ref{Tl} compares the values of the hypergeometric function
up to four decimal places with its interpolating polynomial values in the interval $[0,1]$ 
for the choice of parameters $a=1$, $b=2$ and $c=6$. 
Figure~\ref{Pl} and Table~\ref{Tl} also indicate errors at various points 
within the unit interval except at the end points. 

\begin{figure}[H] 
\includegraphics[width=8cm]{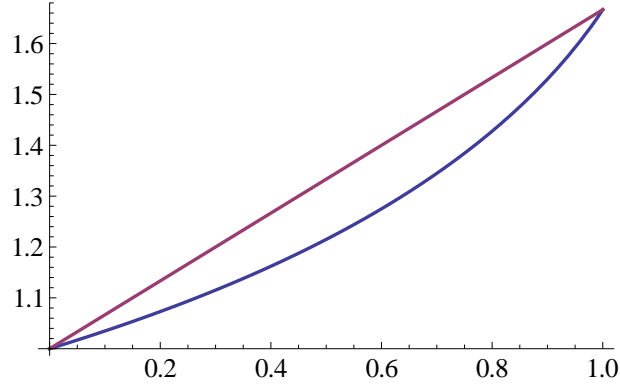}
\caption{Linear interpolation of ${}_2F_1[1,2;6;x]$ at $0$ and $1$.}\label{Pl}
\end{figure}

\begin{table}[H]
\begin{tabular}{|c|c|c|c|c|c|}
\hline 
Nodes $x_i$ & $0$ & $0.25$ & $0.5$ & $0.75$ & $1$ \\
\hline 
Actual values $_2F_1[1,2;6;x_i]$ & $1$ & $1.0936$ & $1.2149$ & $1.3843$ & $1.6667$ \\
\hline 
Polynomial approximations  & $1$ & $1.1667$ & $1.3333$ & $1.5000$ & $1.6667$\\ 
by $P_l(x_i)$ &&&&&\\
\hline
Validity of error bounds & $0$ & $0.0731<1.25$ & $0.1184<1.25$ & $0.1157<1.25$ & $0$\\ 
 by $E_l(f,x_i)$ &&&&&\\
\hline
\end{tabular}
\caption{Comparison of the functional and linear polynomial values}
\end{table}\label{Tl}

\section{Quadratic Interpolation on ${}_2F_1[a,b;c;x]$}
Let the three points in consideration for quadratic interpolation be $x_0=0$, $x_1=0.5$ and $x_2=1$.
The functional values at $x_0=0$ and $x_2=1$ can be found easily in terms of the
parameters but the functional value at $x_1=0.5$ can be obtained through 
different identities involving hypergeometric functions $_2F_1[a,b;c;x]$
dealing with various constraints on the parameters $a,b,c$. 
This section consists of two subsections and in each subsection 
the method to obtain the functional value of $_2F_1[a,b;c;x]$ at $x=0.5$ uses three different identities. 
Finally, we compare the resultant interpolations. In fact we observe that the interpolating polynomial
remains unchanged in two cases although the
approaches are different (see Section~\ref{sec3.2} for more details).

\subsection{Quadratic Interpolation on ${}_2F_1[a,1-a;c;x]$}\label{sec3.1}
This section deals with the value ${}_2F_1[a,b;c;1/2]$ where $a+b=1$ due to the following identity
of Bailey (see \cite[p.~11]{Bai35} and also \cite[p.~69]{RV60}):
\begin{equation}\label{3.1eq1}
{}_2F_1[a,1-a;c;{1}/{2}]=\frac{2^{1-c}\,\Gamma(c)\Gamma(\frac{1}{2})}{\Gamma(\frac{1}{2}(c+a))\,\Gamma(\frac{1}{2}(1+c-a))}
=\frac{\Gamma(\frac{1}{2}c)\,\Gamma(\frac{1}{2}(1+c))}{\Gamma(\frac{1}{2}(c+a))\,\Gamma(\frac{1}{2}(1+c-a))},
\end{equation}
where $c$ is neither zero nor negative integers. 
It follows from \eqref{3.1eq1} that
\begin{equation}\label{3.1eq2}
\Gamma\Big(\frac{1}{2}c\Big)\,\Gamma\Big(\frac{1}{2}(1+c)\Big)=2^{1-c}\,\sqrt{\pi}\,\Gamma(c),
\end{equation}
since $\Gamma(1/2)=\sqrt{\pi}$.
In this case, we obtain
$$
f(x_0)=f(0)={}_2F_1[a,1-a;c;0]=1;
$$
$$
f(x_1)=f(0.5)={}_2F_1[a,1-a;c;1/2]=\frac{\Gamma(\frac{1}{2}c)\,\Gamma(\frac{1}{2}(1+c))}{\Gamma(\frac{1}{2}(c+a))\,\Gamma(\frac{1}{2}(1+c-a))};
$$ 
and 
$$
f(x_2)=f(1)={}_2F_1[a,1-a;c;1]=\frac{\Gamma(c)\Gamma(c-1)}{\Gamma(c-a)\Gamma(c+a-1)}\quad (c>1).
$$
Consider the well-known Lagrange fundamental polynomials 
$$L_0(x)=\frac{(x-x_1)(x-x_2)}{(x_0-x_1)(x_0-x_2)},~~
L_1(x)=\frac{(x-x_0)(x-x_2)}{(x_1-x_0)(x_1-x_2)},~~
L_2(x)=\frac{(x-x_0)(x-x_1)}{(x_2-x_0)(x_2-x_1)}.
$$
Thus, the quadratic interpolation of $f(x)={}_2F_1[a,1-a;c;x]$ becomes 
\begin{align*} 
P_{q_3}(x) 
& = f(x_0)L_0(x)+f(x_1)L_1(x)+f(x_2)L_2(x)\\
& = (2x^2-3x+1)+(-4x^2+4x)\frac{\Gamma(\frac{1}{2}c)\,\Gamma(\frac{1}{2}(1+c))}{\Gamma(\frac{1}{2}(c+a))\,\Gamma(\frac{1}{2}(1+c-a))}\\
& \hspace*{8cm} +(2x^2-x)\frac{\Gamma(c)\Gamma(c-1)}{\Gamma(c-a)\Gamma(c+a-1)}.
\end{align*}
This leads to the following result.

\begin{theorem}\label{3.1thm1}
Let $a,b,c\in\mathbb{R}$ be such that $c>1$.
Then 
\begin{align*}
P_{q_1}(x) &= \left(2-\frac{4\,\Gamma(\frac{1}{2}c)\,\Gamma(\frac{1}{2}(1+c))}{\Gamma(\frac{1}{2}(c+a))\,\Gamma(\frac{1}{2}(1+c-a))}
+\frac{2\,\Gamma(c)\Gamma(c-1)}{\Gamma(c-a)\Gamma(c+a-1)}\right)x^2\\
& \hspace*{2cm}+\left(\frac{4\,\Gamma(\frac{1}{2}c)\,\Gamma(\frac{1}{2}(1+c))}{\Gamma(\frac{1}{2}(c+a))\,\Gamma(\frac{1}{2}(1+c-a))}
-\frac{\Gamma(c)\Gamma(c-1)}{\Gamma(c-a)\Gamma(c+a-1)}-3\right)x+1.
\end{align*} 
is a quadratic interpolation of ${}_2F_1[a,1-a;c;x]$ in $[0,1]$.
\end{theorem}

\begin{remark}
It is evident that when $a=0,1$, then $P_{q_1}(x)={}_2F_1[a,1-a;c;x]=1$ for all $x\in [0,1]$ 
and for all $c>1$. Moreover, for all $c>1$, we have the following three natural observations 
\begin{enumerate}
\item[(i)] if $-1<a<0$, then $P_{q_1}(x)$ and ${}_2F_1[a,1-a;c;x]$ both decrease together in $[0,1]$;
\item[(ii)] if $0<a<1$, then $P_{q_1}(x)$ and ${}_2F_1[a,1-a;c;x]$ both increase together in $[0,1]$; and 
\item[(iii)] if $1<a<2$, then $P_{q_1}(x)$ and ${}_2F_1[a,1-a;c;x]$ both decrease together in $[0,1]$.
\end{enumerate}
Indeed, all of them follow from derivative test. More observations are stated later while 
estimating the error (see Remark~\ref{3.2rem1}).
\end{remark}

An interpolating polynomial $P_{q_1}(x)$ of ${}_2F_1[a,1-a;c;x]$ for certain choices of
parameters $a$ and $c$ is as shown in Figure~\ref{Pq1}.

\begin{figure}[H] 
\centering
\includegraphics[width=8cm]{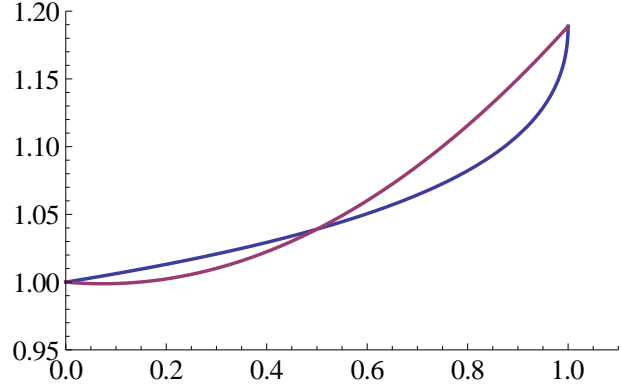}
\caption{The quadratic interpolation of ${}_2F_1[0.9,0.1;1.5;x]$ at $0,0.5$, and $1$.} \label{Pq1}
\end{figure}

\begin{remark}
Note that in Theorem~\ref{3.1thm1}, the parameter $c$ can not be chosen such that 
$c\le (a+b+1)/2$ since the choice $b=1-a$ resulting to $c\le 1$, which is a contradiction to the 
assumption that $c>1$. In particular, $c\neq (a+b+1)/2$ in Theorem~\ref{3.1thm1},
the negation of a constraint that will be considered in the next subsection.
\end{remark}

\subsection{Quadratic Interpolation on ${}_2F_1[a,b;(a+b+1)/2;x]$}\label{sec3.2}
In this section, $f(x)={}_2F_1[a,b;c;x]$, $c=(a+b+1)/2$, is first interpolated using 
the following quadratic transformation obtained from \cite[(3.1.3)]{AAR99} 
(see also \cite[Theorem~2.5]{RV60}).
\begin{lemma}\label{qt1}
If $(a+b+1)/2$ is neither zero nor a negative integer, and if $|x|<1$ and
$|4x(1-x)|<1$, then
\begin{equation}\label{eq4}
_2F_1\left[a,b;\frac{a+b+1}{2};x\right]
= {_2F}_1\left[\frac{a}{2},\frac{b}{2};\frac{a+b+1}{2};4x(1-x)\right].
\end{equation}
\end{lemma}
If we choose $x=0.5$ then the right hand side of (\ref{eq4}) computes
asymptotic behavior of the hypergeometric function at $1$. Hence the functional value at 
$x=0.5$ of the function $f(x)={}_2F_1[a,b;(a+b+1)/2;x]$ can be obtained with the help of 
\eqref{eq2}. Due to Lemma~\ref{qt1} and \eqref{eq2}, in this case, the constraints on the 
parameters are computed as:
\begin{itemize}
\item $a+b<1$;
\item $a+b\neq -(2n+1)$ for $n \in \mathbb{N}\cup\{0\}$.
\end{itemize}
One can easily obtain that
\begin{align*}
f(x_0) &={}_2F_1\left[a,b;\frac{a+b+1}{2};0\right]=1;\\
f(x_1) &= {}_2F_1\left[a,b;\frac{a+b+1}{2};\frac{1}{2}\right]=\frac{\sqrt{\pi}\,\Gamma\Big(\displaystyle\frac{a+b+1}{2}\Big)}{\Gamma\Big(\displaystyle\frac{a+1}{2}\Big)\Gamma\Big(\displaystyle\frac{b+1}{2}\Big)}
;\\
f(x_2) &= {}_2F_1\left[a,b;\frac{a+b+1}{2};1\right]=\frac{\Gamma\Big(\displaystyle\frac{1-a-b}{2}\Big)\Gamma\Big(\displaystyle\frac{a+b+1}{2}\Big)}{\Gamma\Big(\displaystyle\frac{a+1-b}{2}\Big)\Gamma\Big(\displaystyle\frac{b+1-a}{2}\Big)}=\frac{\cos\pi\displaystyle\frac{(b-a)}{2}}{\cos\pi\displaystyle\frac{(b+a)}{2}},
\end{align*}
where 
$f(x_2)$ is obtained by the well-known Euler's reflection formula 
(in non-integral variable $x$)
$$\Gamma(x)\Gamma(1-x)=\frac{\pi}{\sin(\pi x)}.
$$
This leads to the 
additional constraints on the parameters as (these constraints
may be relaxed when one does not use Euler's reflection formula!)

\begin{equation}\label{3.2eq2}
\left\{
\begin{array}{ll}
& a+b\neq 1\pm 2n \quad \mbox{and} \quad a-b\neq -1\pm 2n, ~n\in \mathbb{Z}; \mbox{ or}\\
& a+b\neq -1\pm 2n \quad \mbox{and} \quad a-b\neq 1\pm 2n, ~n\in \mathbb{Z}.
\end{array}
\right.
\end{equation}

Thus, the first quadratic interpolation of $f(x)={}_2F_1[a,b;(a+b+1)/2;x]$ becomes 
\begin{align*} 
P_{q_2}(x) 
& = f(x_0)L_0(x)+f(x_1)L_1(x)+f(x_2)L_2(x)\\
& = (2x^2-3x+1)+(-4x^2+4x)\frac{\sqrt{\pi}\,\Gamma\left(\displaystyle\frac{a+b+1}{2}\right)}
{\Gamma\left(\displaystyle\frac{a+1}{2}\right)\Gamma\left(\displaystyle\frac{b+1}{2}\right)}
+(2x^2-x)\frac{\cos\pi\displaystyle\frac{(b-a)}{2}}{\cos\pi\displaystyle\frac{(b+a)}{2}}.
\end{align*}
This leads to the following result.

\begin{theorem}\label{3.2thm1}
Let $a,b\in\mathbb{R}$ and $n \in \mathbb{N}\cup\{0\}$ be such that $a+b\neq -(2n+1)$ and $a+b<1$.
If either $a+b\neq 1\pm 2n$ and $a-b\neq -1\pm 2n$, or
$a+b\neq -1\pm 2n$ and $a-b\neq 1\pm 2n$ hold,
then 
\begin{align*}
P_{q_2}(x) &= \left(2-\frac{4\sqrt{\pi}\,\Gamma\left(\displaystyle\frac{a+b+1}{2}\right)}
{\Gamma\left(\displaystyle\frac{a+1}{2}\right)\Gamma\left(\displaystyle\frac{b+1}{2}\right)}
+\frac{2\cos\pi\displaystyle\frac{(b-a)}{2}}{\cos\pi\displaystyle\frac{(b+a)}{2}}\right)x^2\\
& \hspace*{3cm}+\left(\frac{4\sqrt{\pi}\,\Gamma\left(\displaystyle\frac{a+b+1}{2}\right)}
{\Gamma\left(\displaystyle\frac{a+1}{2}\right)\Gamma\left(\displaystyle\frac{b+1}{2}\right)}
-\frac{\cos\pi\displaystyle\frac{(b-a)}{2}}{\cos\pi\displaystyle\frac{(b+a)}{2}}-3\right)x+1.
\end{align*} 
is a quadratic interpolation of ${}_2F_1[a,b;(a+b+1)/2;x]$ in $[0,1]$.
\end{theorem}

Secondly, we also discuss quadratic interpolation of the same function
${}_2F_1[a,b;c;x]$, $c=(a+b+1)/2$, in $[0,1]$, but using a different hypergeometric
identity. Finally, we observe that both the interpolations are same
except at a minor difference in one of the constraints.

Recall the transformation formula (see \cite[Theorem~20, p.~60]{RV60}):
\begin{lemma}\label{3.3lem1}
If $|x|<1$ and $|x/(1-x)|<1$, then we have
$$
{}_2F_1[a,b;c;x]={(1-x)^{-a}} {_2F_1[a,c-b;c;\frac{-x}{1-x}]}. 
$$
\end{lemma}

Note that $-x/(1-x)=-1$ for $x=0.5$. To find the value $f(0.5)=2^a{}_2F_1[a,c-b;c;-1]$, this 
suggests us to use the following identity (see \cite[Theorem~26, p~68]{RV60}; see also \cite{BW10}).  

\begin{lemma}\label{3.3lem2}
Let $a',b'\in \mathbb{R}$.
If $1+a'-b'\neq \{0,-1,-2,-3,\ldots\}$ and $b'<1$, then we have
$$
{}_2F_1[a',b';a'-b'+1;-1]=\frac{\Gamma(a'-b'+1)\Gamma\Big(\displaystyle\frac{a'}{2}+1\Big)}
{\Gamma(a'+1)\Gamma\Big(\displaystyle\frac{a'}{2}-b'+1\Big)}.
$$
\end{lemma}

Comparison of the parameters $a'=a$, $b'=c-b$ and $a'-b'+1=c$ leads to
\begin{equation}\label{3.3eq1}
{}_2F_1[a,c-b;c;-1]=\frac{\Gamma(a-c+b+1)\Gamma\Big(\displaystyle\frac{a}{2}+1\Big)}
{\Gamma(a+1)\Gamma\Big(\displaystyle\frac{a}{2}-c+b+1\Big)}
\end{equation}
with the constraints
\begin{itemize}
\item $2c=a+b+1$;
\item $c\neq \{0,-1,-2,-3,\ldots\} \iff a+b\neq -(2n+1),~n\in\mathbb{N}\cup\{0\}$;
\item $c-b<1 \iff a-b<1$.
\end{itemize}
Under these conditions, \eqref{3.3eq1} leads to
\begin{align*}
f(x_1)=f(0.5)
& ={}_2F_1\Big[a,b;\frac{a+b+1}{2};\frac{1}{2}\Big] =2^a\frac{\Gamma\Big(\displaystyle\frac{a+b+1}{2}\Big)\Gamma\Big(\displaystyle\frac{a}{2}+1\Big)}
{\Gamma(a+1)\Gamma\Big(\displaystyle\frac{b+1}{2}\Big)}\\
& =\frac{2^{a-1}\,\Gamma\Big(\cfrac{a+b+1}{2}\Big)\,\Gamma\Big(\cfrac{a}{2}\Big)}{\Gamma(a)\,\Gamma\Big(\cfrac{b+1}{2}\Big)}
=\frac{\sqrt{\pi}\,\Gamma\Big(\cfrac{a+b+1}{2}\Big)}{\Gamma\Big(\cfrac{a+1}{2}\Big)
\Gamma\Big(\cfrac{b+1}{2}\Big)},
\end{align*}
where the last equality holds by \eqref{3.1eq2}.
Also as discussed in Section~\ref{sec3.2}, we have 
$$f(x_0)=f(0)={}_2F_1\Big[a,b;\frac{a+b+1}{2};0\Big]=1,
$$ 
and 
$$f(x_2)=f(1)={}_2F_1\Big[a,b;\frac{a+b+1}{2};1\Big]
=\frac{\cos\pi\displaystyle\frac{(b-a)}{2}}{\cos\pi\displaystyle\frac{(b+a)}{2}},
\quad a+b<1
$$ 
with additional constraints obtained in \eqref{3.2eq2} (here also \eqref{3.2eq2} may be relaxed!).

Thus, the second quadratic interpolation of $f(x)={}_2F_1[a,b;(a+b+1)/2;x]$ remains same as
the first quadratic interpolation obtained in Theorem~\ref{3.2thm1} but with an additional constraint
$a-b<1$. This shows that the quadratic interpolation obtained by Theorem~\ref{3.2thm1} is stronger 
than what was discussed so far using Lemma~\ref{3.3lem1} and Lemma~\ref{3.3lem2}.
A quadratic interpolation of ${}_2F_1[a,b;(a+b+1)/2;x]$ is shown in Figure~\ref{Pq2}.  
\begin{figure}[H] 
\includegraphics[width=8cm]{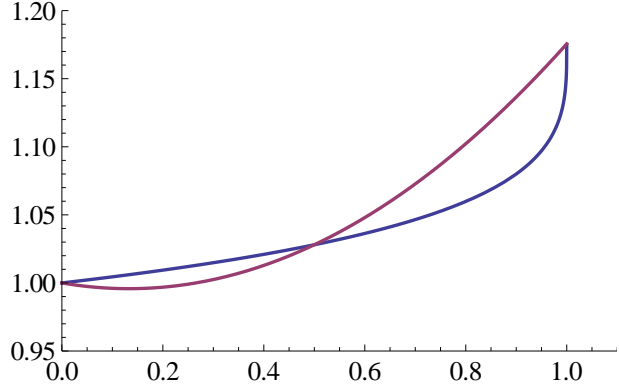}
\caption{The quadratic interpolation of ${}_2F_1[0.1,0.3;0.7;x]$ at $0$, $0.5$, and $1$.}\label{Pq2}
\end{figure}

\subsection{Error Estimates}

The error estimate in quadratic interpolation of ${}_2F_1[a,b;c;x]$
interpolating at $0,0.5,1$ in $[0,1]$ is formulated as below:
\begin{lemma}\label{3.2lem4}
Let $P_q(x)$ be a quadratic interpolation of $f(x)={}_2F_1[a,b;c;x]$
interpolating at $0,0.5,1$ in $[0,1]$. If $a,b,c\in(-3,\infty)$ with $c-a-b>3$,
then the deviation of $f(x)$ from $P_q(x)$ is estimated by
\begin{align*}
|E_q(f,x)|
& =|f(x)-P_q(x)|\\
& \le \frac{M}{6}\,|a(a+1)(a+2)b(b+1)(b+2)|\,
\frac{\Gamma(c)\Gamma(c-a-b-3)}{\Gamma(c-a)\Gamma(c-b)}
\end{align*}
for all values of $x\in[0,1]$, where $M$ is defined by
\begin{align}\label{M}
M & :=\left\{\begin{array}{ll}
\cfrac{1}{12}(3-\sqrt{3})(-1+\cfrac{1}{6}(3-\sqrt{3}))(-1+\cfrac{1}{3}(3-\sqrt{3})), & \mbox{ $x<1/2$,}\\[4mm]
-\cfrac{1}{12}(3+\sqrt{3})(-1+\cfrac{1}{6}(3+\sqrt{3}))(-1+\cfrac{1}{3}(3+\sqrt{3})), & \mbox{ $x>1/2$.}
\end{array}\right.
\end{align}
\end{lemma}
\begin{proof}
It requires to estimate 
$$\max_{0\le x\le 1} \frac{|x(x-0.5)(x-1)|}{6}\max_{0\le x\le 1}|f'''(x)|,
$$
where $f(x)={}_2F_1[a,b;c;x]$. Note that 
$$
\max_{0\le x\le 1} |x(x-0.5)(x-1)|=M ~(\approx 0.0481125\cdots)
$$
obtained by \eqref{M}.
We apply the well-known derivative formula \eqref{der-for} to maximize $|f'''(x)|$, $0\le x\le 1$.
The proof follows from (\ref{eq2}), Lemma~\ref{lem-AVV}, (\ref{der-for}), and the fact that
$\Gamma(x+1)=x\Gamma(x)$. 
\end{proof}

The following result is an immediate consequence of Lemma~\ref{3.2lem4} which estimates the difference
$E_{q_1}(f,x)={}_2F_1[a,1-a;c;x]-P_{q_1}(x)$ in $[0,1]$.

\begin{corollary}\label{3.2cor1}
Let $a,c\in\mathbb{R}$ be such that $-3<a<4$ and $c>4$. 
Then the deviation of ${}_2F_1[a,1-a;c;x]$ from $P_{q_1}(x)$ is estimated by
\begin{align*}
|E_{q_1}(f,x)|
& =|f(x)-P_{q_1}(x)|\\
& \le \frac{M}{6}\,|a(a+1)(a+2)(1-a)(2-a)(3-a)|\,
\frac{\Gamma(c)\Gamma(c-4)}{\Gamma(c-a)\Gamma(c+a-1)}
\end{align*}
for all values of $x\in[0,1]$, where $M$ is obtained by \eqref{M}.
\end{corollary}

\begin{remark}\label{3.2rem1}
It follows from Corollary~\ref{3.2cor1} that there is no error for either of the choices 
$a=-2,-1,0,1,2,3$. In other words, for either of these choices,
$E_{q_1}(f,x)$ vanishes. 
\end{remark}

Similarly, as a consequence of Lemma~\ref{3.2lem4}, we obtain

\begin{corollary}\label{3.2cor2}
Let $a,b\in\mathbb{R}$ be such that $-7<a+b<-5$. 
Then the deviation of ${}_2F_1[a,b;(a+b+1)/2;x]$ from $P_{q_2}(x)$ is estimated by
\begin{align*}
|E_{q_2}(f,x)|
& =|f(x)-P_{q_2}(x)|\\
& \le \frac{M}{6}\,|a(a+1)(a+2)b(b+1)(b+2)|\,
\frac{\Gamma\Big(\cfrac{a+b+1}{2}\Big)\Gamma\Big(\cfrac{-a-b-5}{2}\Big)}
{\Gamma\Big(\cfrac{b-a+1}{2}\Big)\Gamma\Big(\cfrac{a-b+1}{2}\Big)}
\end{align*}
for all values of $x\in[0,1]$, where $M$ is obtained by \eqref{M}.
\end{corollary}

\begin{remark}\label{3.2rem2}
It follows from Corollary~\ref{3.2cor2} that 
since $E_{q_2}(f,x)$ vanishes for the choices $a=-2,-1,0$ and $b=-2,-1,0$, 
there is no error for these choices of the parameters $a$ and $b$.
\end{remark}

Now we describe a bit deeper analysis on the error obtained in Corollary~\ref{3.2cor1} through the following lemma
which is a consequence of Lemma~\ref{lem-AVV-p36}. A similar analysis can be described for Corollary~\ref{3.2cor2}.

\begin{lemma}\label{3.2lem1}
Let $a,c\in\mathbb{R}$ be such that $c>4$. If either $1<a<4$ or $-3<a<0$ holds, then the quotient
$$
\frac{\Gamma(c)\Gamma(c-4)}{\Gamma(c-a)\Gamma(c+a-1)}
$$
decreases when $c$ increases.
\end{lemma}
\begin{proof}
We use Lemma~\ref{lem-AVV-p36}. Since $c-a>c-4>0$, in one hand we have
$$
\frac{\Gamma'(c-4)}{\Gamma(c-4)}-\frac{\Gamma'(c-a)}{\Gamma(c-a)}<0.
$$
On the other hand, since $c<c+a-1$, we have
$$
\frac{\Gamma'(c)}{\Gamma(c)}-\frac{\Gamma'(c+a-1)}{\Gamma(c+a-1)}<0.
$$
Thus, if 
$$
g(c)=\frac{\Gamma(c)\Gamma(c-4)}{\Gamma(c-a)\Gamma(c+a-1)},
$$
it follows that
\begin{align*}
\frac{g'(c)}{g(c)}
& =\frac{\Gamma'(c)}{\Gamma(c)}+\frac{\Gamma'(c-4)}{\Gamma(c-4)}
-\frac{\Gamma'(c-a)}{\Gamma(c-a)}-\frac{\Gamma'(c+a-1)}{\Gamma(c+a-1)}\\
& = \left(\frac{\Gamma'(c-4)}{\Gamma(c-4)}
-\frac{\Gamma'(c-a)}{\Gamma(c-a)}\right)+
\left(\frac{\Gamma'(c)}{\Gamma(c)}-\frac{\Gamma'(c+a-1)}{\Gamma(c+a-1)}\right)\\
& < 0.
\end{align*}
By the definition of the gamma function, obviously, one can see that $\Gamma(x)>0$ for $x>0$.
This shows that $g(c)>0$ and hence $g'(c)<0$. Thus, $g(c)$ decreases for $1<a<4<c$.

For $c>4$, if $-3<a<0$ holds then we consider the rearrangement
$$
\frac{g'(c)}{g(c)}=\left(\frac{\Gamma'(c)}{\Gamma(c)}
-\frac{\Gamma'(c-a)}{\Gamma(c-a)}\right)+\left(\frac{\Gamma'(c-4)}{\Gamma(c-4)}-\frac{\Gamma'(c+a-1)}{\Gamma(c+a-1)}\right)
$$
and show that ${g'(c)}/{g(c)}<0$.
\end{proof}

Using Mathematica or other similar tools, one can see that Lemma~\ref{3.2lem1} even holds true
for the remaining range $0\le a\le 1$. This suggests us to pose the following conjecture.

\begin{conjecture}\label{3.2conj}
Let $a,c\in\mathbb{R}$ be such that $0\le a\le 1$ and $c>4$. Then the quotient
$$
\frac{\Gamma(c)\Gamma(c-4)}{\Gamma(c-a)\Gamma(c+a-1)}
$$
decreases when $c$ increases.
\end{conjecture}

Thus, we observe that when $c>4$ increases then the error $E_{q_1}(f,x)$ estimated in 
Corollary~\ref{3.2cor1} decreases (see also Figure~\ref{Eq1fx-1} and Figure~\ref{Eq1fx-2}).

\begin{figure}[H] 
\begin{center}
\includegraphics[width=7cm]{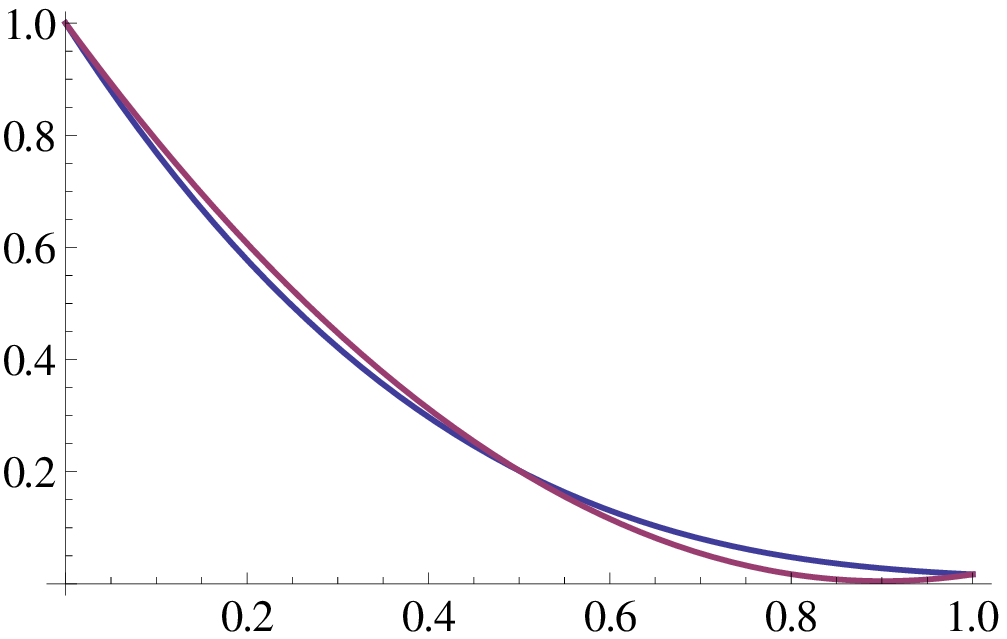}
\includegraphics[width=7cm]{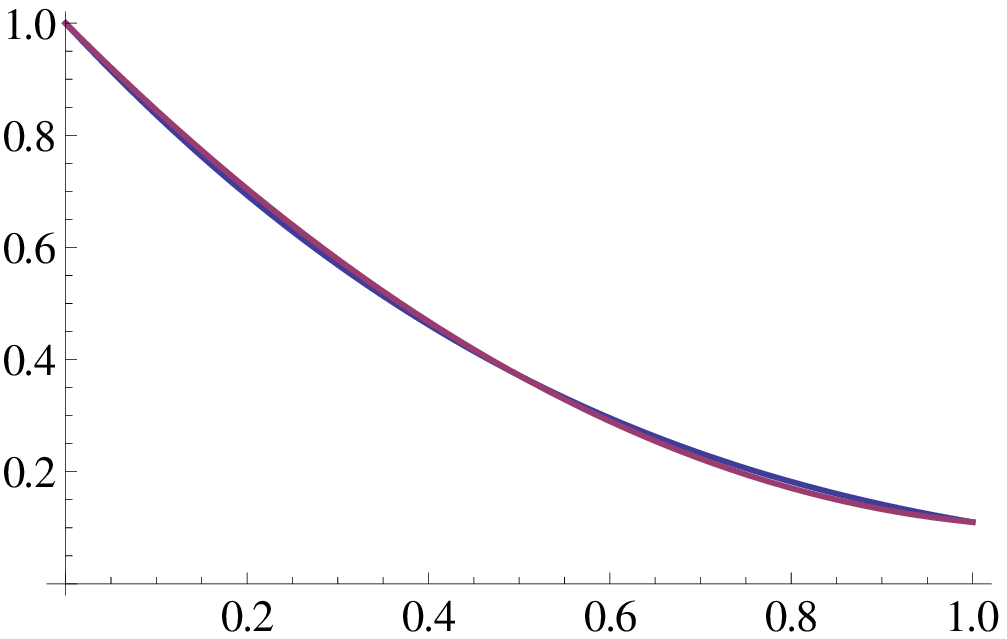}
\end{center}
\caption{The error estimate $E_{q_1}(f,x)$ when $a=3.9$ and $c$ increases from $4.5$ to $6.5$.}\label{Eq1fx-1}
\end{figure}

\begin{figure}[H] 
\begin{center}
\includegraphics[width=7cm]{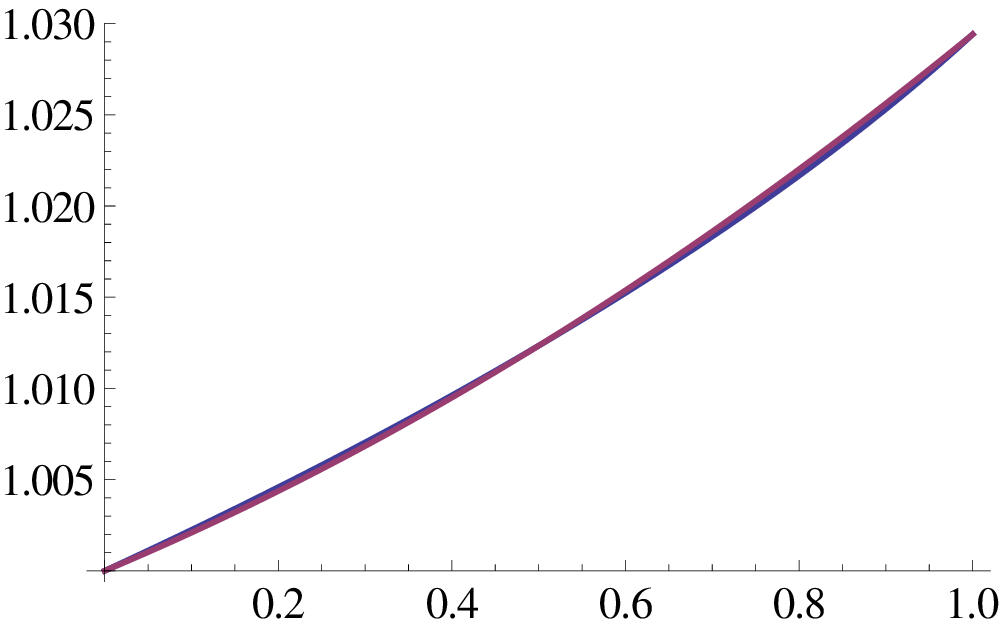}
\includegraphics[width=7cm]{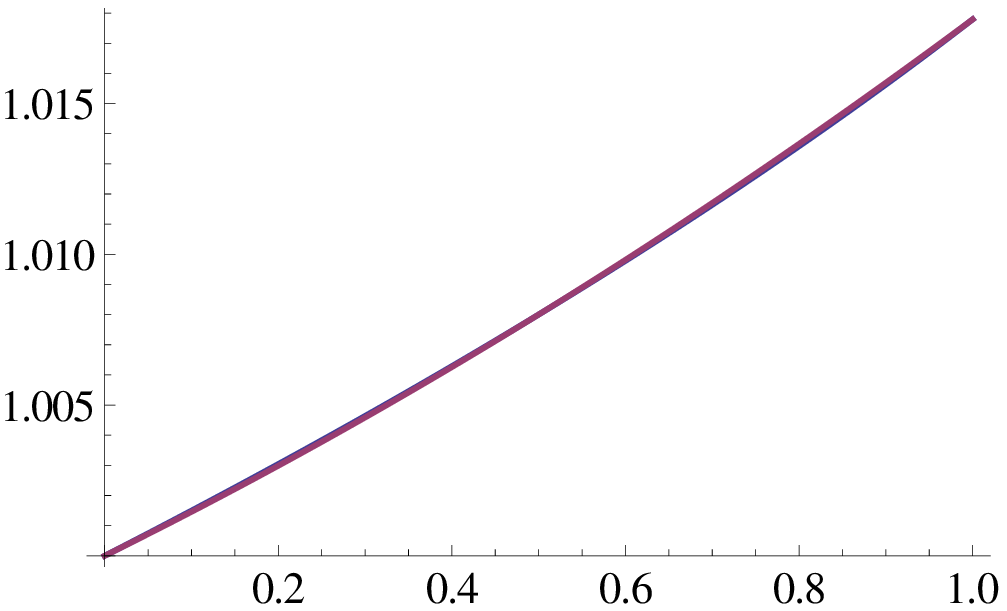}
\end{center}
\caption{The error estimate $E_{q_1}(f,x)$ when $a=0.9$ and $c$ increases from $4.1$ to $6.1$.}\label{Eq1fx-2}
\end{figure}

\begin{table}[H]
\begin{tabular}{|c|c|c|c|c|c|}
\hline 
Nodes $x_i$ & $0$ & $0.25$ & $0.5$ & $0.75$ & $1$ \\
\hline 
Actual values ${}_2F_1[3.9,-2.9;5;x_i]$ & $1$ & $0.5372$ & $0.2516$ & $0.0998$ & $0.0367$ \\
\hline 
Polynomial approximations  & $1$ & $0.5591$ & $0.2516$ & $0.0775$ & $0.0367$\\ 
by $P_{q_1}(x_i)$ &&&&&\\
\hline
Validity of error bounds & $0$ & $0.0219<0.0274$ & $0$ & $0.0223<0.0274$ & $0$\\ 
 by $E_{q_1}(f,x_i)$ &&&&&\\
\hline
&&&&&\\[-4.5mm]
\hline
Actual values ${}_2F_1[3.9,-2.9;6;x_i]$ & $1$ & $0.6027$ & $0.3358$ & $0.1724$ & $0.0845$ \\
\hline 
Polynomial approximations  & $1$ & $0.6163$ & $0.3358$ & $0.1585$ & $0.0845$\\ 
by $P_{q_1}(x_i)$ &&&&&\\
\hline
Validity of error bounds & $0$ & $0.0136<0.0158$ & $0$ & $0.0139<0.0158$ & $0$\\ 
 by $E_{q_1}(f,x_i)$ &&&&&\\
\hline
\end{tabular}
\caption{Comparison of the functional and quadratic polynomial values}\label{T2}
\end{table}

\begin{table}[H]
\begin{tabular}{|c|c|c|c|c|c|}
\hline 
Nodes $x_i$ & $0$ & $0.25$ & $0.5$ & $0.75$ & $1$ \\
\hline 
Actual values ${}_2F_1[0.9,0.1;5;x_i]$ & $1$ & $1.0047$ & $1.0099$ & $1.0158$ & $1.0227$ \\
\hline 
Polynomial approximations  & $1$ & $1.0046$ & $1.0099$ & $1.0160$ & $1.0227$\\ 
by $P_{q_1}(x_i)$ &&&&&\\
\hline
Validity of error bounds & $0$ & $0.0001<0.0016$ & $0$ & $0.0002<0.0016$ & $0$\\ 
 by $E_{q_1}(f,x_i)$ &&&&&\\
\hline
&&&&&\\[-4.5mm]
\hline
Actual values ${}_2F_1[0.9,0.1;6;x_i]$ & $1$ & $1.0039$ & $1.0082$ & $1.0128$ & $1.0182$ \\
\hline 
Polynomial approximations  & $1$ & $1.0038$ & $1.0082$ & $1.0129$ & $1.0182$\\ 
by $P_{q_1}(x_i)$ &&&&&\\
\hline
Validity of error bounds & $0$ & $0.0001<0.0004$ & $0$ & $0.0001<0.0004$ & $0$\\ 
 by $E_{q_1}(f,x_i)$ &&&&&\\
\hline
\end{tabular}
\caption{Comparison of the functional and quadratic polynomial values}\label{T3}
\end{table}

Figure~\ref{Eq1fx-1} and Figure~\ref{Eq1fx-2} describe the quadratic interpolation of the hypergeometric functions 
${}_2F_1[a,1-a,c,x]$ at $0$, $0.5$ and $1$, whereas
Table~\ref{T2} and Table~\ref{T3} respectively compare the values of the hypergeometric functions
up to four decimal places with its interpolating polynomial values in the interval $[0,1]$ 
for the choice of parameters $a=3.9$ and $a=0.9$, $c=5$ and $c=6$. 
Figures~\ref{Eq1fx-1}--\ref{Eq1fx-2} and Tables~\ref{T2}--\ref{T3} also indicate errors at various points 
within the unit interval except at the interpolating points at $x=0,0.5,1$. 

The error estimate $|E_{q_2}(f,x)|$ for the function ${}_2F_1[a,b;(a+b+1)/2;x]$ can be analyzed 
in a similar way, and hence we omit the proof.

\section{An Application}
In this section, we brief on interpolation of a continued fraction 
that converges to a quotient of two hypergeometric functions. 
Gauss used the contiguous relations to give several ways to write a quotient 
of two hypergeometric functions as a continued fraction. For instance, it is well-known that
\begin{equation}\label{cf}
\frac{{}_2F_1[a+1,b;c+1;x]}{{}_2F_1[a,b;c;x]}
= \cfrac{1}{1 + \cfrac{\cfrac{(a-c)b}{c(c+1)} x}{1 + \cfrac{\cfrac{(b-c-1)(a+1)}{(c+1)(c+2)} x}{1 
+ \cfrac{\cfrac{(a-c-1)(b+1)}{(c+2)(c+3)} x}
{1 + \cfrac{\cfrac{(b-c-2)(a+2)}{(c+3)(c+4)} x}{1 + {}\ddots}}}}}, \quad |x|<1.
\end{equation}

In one hand, if we adopt the basic linear interpolation method that we discussed in Section~2 
(that is, linear interpolation directly) to the function
$$
g(x)=\frac{{}_2F_1[a+1,b;c+1;x]}{{}_2F_1[a,b;c;x]}
$$
at $x_0=0$ and $x_1=1$, we obtain the linear interpolation 
of the above continued fraction in the following form:
$$
R_l(x)=g(x_0)+\frac{x-x_0}{x_1-x_0}(g(x)-g(x_0))=1+\Big(\frac{b}{c-b}\Big)x, \quad c-b>a,
$$ 
since $g(x_0)=1$ and $g(x_1)=c/(c-b)$. For the choice $a=1,b=2,c=6$, this approximation is also shown in Figure~\ref{Rl}.
\begin{figure}[H] 
\includegraphics[width=8cm]{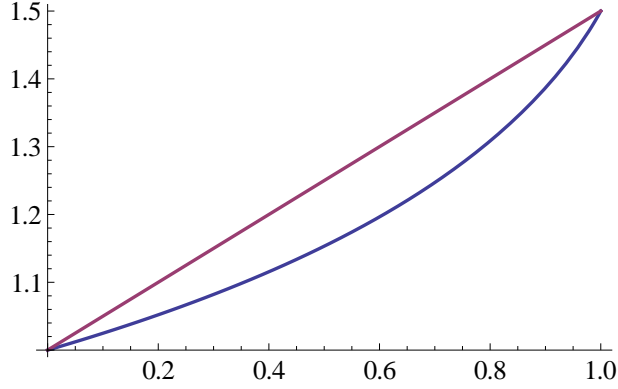}
\caption{Approximation of ${{}_2F_1[a+1,b;c+1;x]}/{{}_2F_1[a,b;c;x]}$ through $R_l(x)$.}\label{Rl}
\end{figure}

On the other hand, an application of linear interpolation of ${}_2F_1[a,b;c;x]$ obtained in Section~\ref{sec2}
leads to the following approximation of the above continued fraction
in terms of ratio of polynomial approximation (we call this {\em rational interpolation}):
\begin{align*}
R_r(x)
&=\frac{1}{P_l(x)}\left(\frac{\Gamma(c+1) \Gamma(c-a-b) -\Gamma(c-a) \Gamma(c-b+1)}
{\Gamma(c-a) \Gamma(c-b+1)}x+1 \right)\\
&=\frac{\Big[\cfrac{c\Gamma(c)\Gamma(c-a-b)}{c-b}-\Gamma(c-a)\Gamma(c-b)\Big]x+\Gamma(c-a)\Gamma(c-b)}
{[\Gamma(c)\Gamma(c-a-b)-\Gamma(c-a)\Gamma(c-b)]x+\Gamma(c-a)\Gamma(c-b)}\\
& =1+\frac{b}{c-b} \left[\frac{\Gamma(c-a-b) \Gamma(c) \,x}{[\Gamma(c) \Gamma(c-a-b)
-\Gamma(c-a) \Gamma(c-b)] \,x+ \Gamma(c-a) \Gamma(c-b)}\right],
\end{align*}
where $c-a-b>0$. For the choice $a=1,b=2,c=6$, this approximation is also shown in Figure~\ref{Rr}.
\begin{figure}[H] 
\includegraphics[width=8cm]{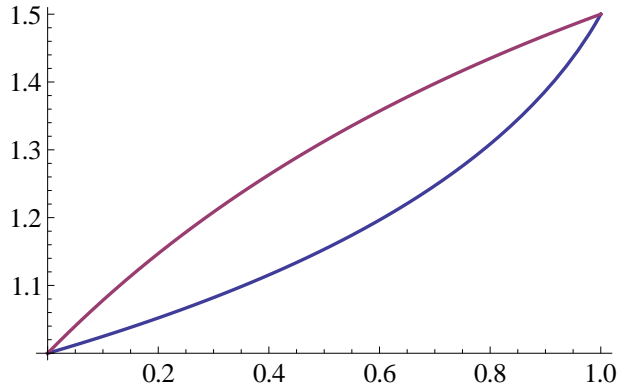}
\caption{Approximation of ${{}_2F_1[a+1,b;c+1;x]}/{{}_2F_1[a,b;c;x]}$ through $R_r(x)$.}\label{Rr}
\end{figure}

Observe that 
$$R_r(x_0)=1=R_l(x_0)~~\mbox{ and }~~R_r(x_1)=\frac{c}{c-b}=R_l(x_1)
$$
and hence $R_r$ also interpolates the continued fraction under consideration at $0$ and $1$. 
Further we observe that both the approximations $R_l(x)$ and $R_r(x)$ of the continued fraction 
are easy to obtain and the first approximation
(i.e., $R_l(x)$) is in a simpler form than $R_r(x)$ as expected.
Now, it would be interesting to know which one would give the best approximation to the continued 
fraction under consideration. With the special choice $a=1,b=2,c=6$, we see from
Figure~\ref{Rl} and Figure~\ref{Rr} that among these two, $R_l(x)$ is the better approximation than 
$R_r(x)$. One may ask: does it happen for arbitrary parameters $a,b,c$?
Since $R_l(x)=R_r(x)$ if and only if $\Gamma(c)\Gamma(c-a-b)=\Gamma(c-a)\Gamma(c-b)$,
the answer to this affirmative question is yes except when $\Gamma(c)\Gamma(c-a-b)=\Gamma(c-a)\Gamma(c-b)$.

This leads to the following result:

\begin{theorem}
Let $R_l(x)$ and $R_r(x)$ be respectively the linear interpolation and the rational interpolation
of the quotient ${{}_2F_1[a+1,b;c+1;x]}/{{}_2F_1[a,b;c;x]}$ (euivalently, of the continued fraction
\eqref{cf}). Then $R_l(x)$ and $R_r(x)$ coincide each other if and only if 
$\Gamma(c)\Gamma(c-a-b)=\Gamma(c-a)\Gamma(c-b)$ holds for $c-a-b>0$.
\end{theorem}

\section{Concluding Remarks and Future Scope}
Recall that, in this paper, we use some standard interpolation techniques to approximate the hypergeometric function 
$$
{}_2F_1[a,b;c;x]=1+\frac{ab}{c}x+\frac{a(a+1)b(b+1)}{c(c+1)}\frac{x^2}{2!}+\cdots
$$
for a range of parameter triples $(a,b,c)$ on the interval $0<x<1$. 
Some of the familiar hypergeometric functional identities and asymptotic behavior of the hypergeometric function 
at $x=1$ played crucial roles in deriving the formula for such approximations.
One can expect similar formulae using other well-known interpolations and obtain better approximation
for the hypergeometric function, however, we discuss such results in the upcoming manuscript(s). 
Different numerical methods for the computation of the confluent and Gauss hypergeometric functions
are studied recently in \cite{POP17}. Such investigation may be extended to the $q$-analog of 
the hypergeometric functions, namely, Heine's basic hypergeometric functions; for instance refer to \cite{CF11}
for similar discussions.

We also focus on error analysis of the numerical approximations leading to monotone properties of 
quotient of gamma functions in parameter triples $(a,b,c)$. Monotone properties of the gamma and its quotients
in different forms
are of recent interest to many researchers; see for instant \cite{Alz93,AQ97,BI86,CZ14,GL01,Gau59,LWZ17,MD17}. 
In this paper, we also studied and 
stated a conjecture (see Conjecture~\ref{3.2conj}) related to monotone properties of quotient of
gamma functions to analyse the error estimate of the numerical approximations under consideration.

Finally, an application to continued fractions of Gauss is also discussed. Approximations of continued fractions
in different forms are also attracted to many researchers; see \cite{LLQ17,LSM16} and references therein 
for some of the recent works.

\bigskip
\noindent
{\bf Acknowledgement.} This work was carried out when the first author was in internship 
at IIT Indore during the summer 2014. 
The authors would like to thank the referee and the editor
for their valuable remarks on this paper.


\begin{thebibliography}{99}
\bibitem{Alz93}
H. Alzer, 
{\em Some gamma function inequalities,}
Math. Comput., {\bf 60} (1993), 337--346.

\bibitem{AQ97}
G. D. Anderson and S.-L. Qui,
{\em A monotoneity property of the gamma function,}
Proc. Amer. Math. Soc., {\bf 125} (1997), 3355--3362.

\bibitem{AVV97}
G. D. Anderson, M. K. Vamanamurthy, and M. K. Vuorinen,
{\em Conformal Invariants, Inequalities, and Quasiconformal Maps,}
John Wiley and Sons, New York (1997).

\bibitem{AAR99} 
G. E. Andrews, R. Askey, and R. Roy,
\emph{Special Functions},
Cambridge University Press, Cambridge (1999).

\bibitem{A89}
K. E. Atkinson, 
{\em An Introduction to Numerical Analysis},
John Wiley and Sons, New York (1989).

\bibitem{Bai35}
W. N. Bailey, {\em Generalized Hypergeometric Series}, 
Cambridge University Press, Cambridge (1935).

\bibitem{BW10}  
R. Beals and R. Wong,
\emph{Special Functions}, 
Cambridge Studies in advanced Mathematics, 126, 
Cambridge University Press, Cambridge (2010).

\bibitem{BI86}
J. Bustoz and M. E. H. Ismail,
{\em On gamma function inequalities,}
Math. Comput., {\bf 47} (1986), 659--667.

\bibitem{CF11}
S. H. L. Chen and A. M. Fu,
{\em A $2n$-point interpolation formula with its applications to 
$q$-identities},
Discrete Math., {\bf 311} (2011), 1793--1802.

\bibitem{CZ14}
B. Chen and H. Zhou,
{\em On completely monotone of an arbitrary real parameter function involving the gamma function,}
Appl. Math. Comput., {\bf 242} (2014), 658--663.

\bibitem{GL01}
C. Giordano and A. Laforgia,
{\em Inequalities and monotonicity properties for the gamma function,}
J. Comput. Appl. Math., {\bf 133} (2001), 387--396.

\bibitem{Gau59}
W. Gautschi,
{\em Some elementary inequalities relating to the gamma and incomplete gamma functions,}
J. Math. Phy., {\bf 38} (1959), 77--81.

\bibitem{LLQ17}
D. Lu, X. Liu, and T. Qu,
{\em Continued fraction approximations and inequalities for the gamma function by Burnside},
Ramanujan J., {\bf 42} (2017), 491--500.

\bibitem{LSM16}
D. Lu, L. Song, and C. Ma,
{\em A quicker continued fraction approximation of the gamma function 
related to the Windschitl's formula,}
Numer. Algor., {\bf 72} (2016), 865--874.

\bibitem{LWZ17}
S. Luo, J. Wei, and W. Zou,
{\em On a transcendental equation involving quotients of gamma functions},
Proc. Amer. Math. Soc., {\bf 145} (6) (2017), 2623--2637.

\bibitem{MD17}
C. Mortici and S. Dumitrescu,
{\em Efficient approximations of the gamma function and further properties,}
Comp. Appl. Math., {\bf 36} (2017), 677--691.

\bibitem{POP17}
J. W. Pearson, S. Olver, and M. A. Porter,
{\em Numerical methods for the computation of the confluent and Gauss hypergeometric functions,}
Numer. Algor., {\bf 74} (2017), 821--866.

\bibitem{RV43} 
E. D. Rainville,
\emph{Intermediate Differential Equations},
John Wiley and Sons, New York (1943).

\bibitem{RV60} 
E. D. Rainville,
\emph{Special Functions},
The Macmillan Company, New York (1960).
\end{thebibliography}
\end{document}